\documentclass[12pt]{article}

\newif\ifarXiv
\arXivtrue

\ifarXiv
\usepackage{e-jc_new}
\else
\usepackage{e-jc}
\fi

\usepackage{amsthm,amsmath,amssymb}

\usepackage{mathtools}
\usepackage{rotating}

\usepackage{graphicx}

\usepackage[colorlinks=true,citecolor=black,linkcolor=black,urlcolor=blue]{hyperref}

\theoremstyle{plain}
\newtheorem{theorem}{Theorem}[section]
\newtheorem{thm}[theorem]{Theorem}
\newtheorem{Theorem}[theorem]{Theorem}

\newtheorem{lem}[theorem]{Lemma}
\newtheorem{Lemma}[theorem]{Lemma}

\newtheorem{proposition}[theorem]{Proposition}
\newtheorem{Proposition}[theorem]{Proposition}

\theoremstyle{definition}
\newtheorem{definition}[theorem]{Definition}
\newtheorem{Definition}[theorem]{Definition}
\newtheorem{example}[theorem]{Example}

\newtheorem{openproblem}{Open Problem}

\theoremstyle{remark}

\newtheorem{Remark}[theorem]{Remark}

\newcommand{\Qmat}{\mathbf{Q}}
\newcommand{\Amat}{\mathbf{A}}
\newcommand{\Bmat}{\mathbf{B}}
\newcommand{\Pmat}{\mathbf{P}}
\newcommand{\Dmat}{\mathbf{M}}
\newcommand{\Umat}{\mathbf{U}}
\newcommand{\Vmat}{\mathbf{V}}

\newcommand{\uvec}{\mathbf{u}}

\newcommand{\dvec}{\mathbf{d}}
\newcommand{\Mat}{\mathrm{Mat}}
\newcommand{\p}[1]{p_{#1}}

\newcommand{\nm}[1]{\mbox{\normalsize $#1$}}

\newcommand*\samethanks[1][\value{footnote}]{\footnotemark[#1]}

\newcommand{\triangleOp}{\odot}
\renewcommand{\emptyset}{\varnothing}

\DeclarePairedDelimiter\abs{\lvert}{\rvert}

\newcommand{\ditem}{\item[---]}

\title{\bf Recurrence relations and splitting formulas \\ for the domination polynomial}

\author{Tomer Kotek\thanks{Supported by the Fein foundation and the graduate school of the Technion. }\\
\small Department of Computer Science\\[-0.8ex]
\small Technion---Israel Institute of Technology\\[-0.8ex] 
\small Haifa, Israel\\
\small\tt tkotek@cs.technion.ac.il\\
\and
James Preen\\
\small Mathematics\\[-0.8ex]
\small Cape Breton University\\[-0.8ex]
\small Sydney, NS B1P6L2, Canada\\
\small\tt james\_preen@capebretonu.ca\\
\and
Frank Simon\thanks{Supported by a grant from the European Social Fund. } \qquad Peter Tittmann \qquad Martin Trinks\samethanks \\
\small Faculty Mathematics / Sciences / Computer Sciences \\[-0.8ex]
\small Hochschule Mittweida - University of Applied Sciences \\[-0.8ex]
\small Mittweida, Germany\\
\small\tt \{simon,peter,trinks\}@hs-mittweida.de\\
}

\begin{document}

\maketitle


\begin{abstract}
The domination polynomial $D(G,x)$ of a graph $G$ is the generating function
of its dominating sets. 
We prove that $D(G,x)$ satisfies a wide range of reduction formulas. 
We show linear recurrence relations for $D(G,x)$ for arbitrary graphs and for various special cases. 
We give splitting formulas for $D(G,x)$ based on articulation vertices, and 
more generally, on splitting sets of vertices.

\end{abstract}

\section{Introduction}

Recurrence relations of graph polynomials have received considerable attention in the literature. 
Informally, 
a graph polynomial $f(G, x)$ {\em satisfies a linear recurrence relation} if
\begin{align*}
f(G,x) &= \sum_{i=1}^k g_i(x) f(G_i, x),
\end{align*}
where $G_i$ are obtained from $G$ using various vertex and edge elimination operations and the $g_i(x)$'s are given rational functions.
For example, it is well-known that the independence polynomial satisfies a linear recurrence relation with respect to two vertex elimination operations, 
the deletion of a vertex and the deletion of vertex's closed neighborhood.
Other prominent graph polynomials in the literature
satisfy similar recurrence relations with respect to vertex and edge elimination operations,
among them the matching polynomial, the chromatic polynomial and the vertex-cover polynomial, 
see e.g. \cite{ar:AGM10}.

In contrast, it is significantly harder to find recurrence relations for the domination polynomial. 
We show in Theorem \ref{th:non-exist_v} that $D(G,x)$ does not satisfy any linear recurrence relation which applies 
only the commonly used vertex operations of 
deletion, extraction, contraction and neighborhood-contraction.
Nor does $D(G,x)$ satisfy any linear recurrence relation using only edge deletion, contraction and extraction. 

In spite of this non-existence result, we give in this paper an abundance of recurrence relations and splitting formulas for the domination polynomial.

The domination polynomial was studied recently by several authors, see 
\cite{Akbari2010,ar:AkbariAlikhaniOboudiPeng2010,ar:Alikhani2011,ar:Alikhani2011ARX,ar:AlikhaniPengIntro,ar:AlikhaniPeng2011,Alikhani2009,pathof2,Arocha2000,Dohmen2012}.
The previous research focused mainly on the roots of domination polynomials 
and on the domination polynomials of various classes of special graphs. 
In \cite{Dohmen2012} it is shown that computing 
the domination polynomial $D(G,x)$ of a graph $G$ is NP-hard and
some examples for graphs for which  $D(G,x)$ can be computed efficiently are given. 
Some of our results, e.g. Theorem~\ref{thm:symmDomSplittingFormula}, lead to efficient schemes to compute the 
domination polynomial. 

An outline of the paper is as follows.
In Section \ref{se:arb} we give a recurrence relation for arbitrary graphs. 
In Section \ref{se:special} we give simple recurrence steps in special cases, which allow us e.g. to
dispose of triangles, induced $5$-paths and irrelevant edges. 
In Section \ref{se:art} we consider graphs of connectivity $1$,
and give several splitting formulas for them. In Section \ref{se:sep} we generalize the results of the previous section 
to arbitrary separating vertex sets. 
In Section \ref{se:der} we show a recurrence relation for arbitrary graphs which 
uses derivatives of domination polynomials. 
\subsection{Definitions and notations}

This paper discusses simple undirected graphs $G=(V,E)$. 
A vertex subset $W \subseteq V$ of $G$ is a \emph{dominating
  vertex set} in $G$, if for each vertex $v \in V$ of $G$ either
$v$ itself or an adjacent vertex is in $W$. 
\begin{definition}
Let $G = (V, E)$ be a graph.
The
\emph{domination polynomial} $D(G, x)$ is 
given by 
\[
 D(G,x)=\sum_{i=0}^{|V|} d_i(G) x^i\,,
\]
where $d_i(G)$ is the number of dominating sets of size $i$ in $G$. 
\end{definition}

A graph polynomial $f(G,x)$ is 
{\em multiplicative with respect to components}, if
\begin{align*}
f(G^1 \cup G^2, x) &= f(G^1, x) \cdot f(G^2, x),
\end{align*}
where $G^1\cup G^2$ denotes the disjoint union of the graphs $G^1$ and
$G^2$.
Obviously, the domination polynomial $D(G, x)$ is multiplicative with respect to 
components.

We use the following  notations: The set of vertices adjacent in $G$ to a vertex of a vertex subset $W
\subseteq V$ is the \emph{open neighborhood} $N_G(W)$ of $W$. The \emph{closed neighborhood}
$N_G[W]$ additionally includes all vertices of $W$ itself. In case of a
singleton set $W = \{v\}$ we write $N_G(v)$ and $N_G[v]$ instead
of $N_G(\{v\})$ and $N_G[\{v\}]$, respectively.
We omit the subscript when the graph $G$ is clear from the context. 

Hence, a vertex subset $W \subseteq V$ of $G$ is a dominating set
of $G$, if all vertices are in the closed neighborhood of $W$,
i.e. $N_G[W] = V$, and the domination polynomial can be stated as
a subset expansion:
\begin{align}
D(G, x) = \sum_{W \subseteq V}{[N_G[W] = V] x^{\abs{W}}}\,,
\end{align}
where we denote, for any statement
 $C_W$ on $W$, 
\[
 [C_W] = \begin{cases}
                 1 & \mbox{if } C_W\mbox{ holds for }W,\\	
		 0 & \mbox{if } C_W\mbox{ does not hold for }W\,.
                \end{cases}
\]
We generalize this representation of $D(G,x)$ as follows\footnote{
This definition can be generalized using the logical framework of \cite{ar:MakowskyZoo}. 
}:
\begin{definition} \label{def:gen-dom}
Let $G = (V, E)$ be a graph.
 Let $C=C_W$ be a statement, then we denote by
\(D_C(G, x)\) the
generating function for the number of dominating vertex sets of \(G\)
satisfying $C$:
\[
D_C(G, x) = \sum_{W \subseteq V}{[N_G[W]=V] [C_W] x^{\abs{W}}}.
\]
\end{definition}
Examples for statements $C_W$ are: (a) $W\not=\emptyset$ (b)  $u\in W$, (c) $u\notin W$, and (d) $u\notin N_G[W]$, 
where $u$  in (b), (c) and (d) is some fixed vertex of $G$.


In the sequel we use for graphs $G = (V, E)$ the following vertex and edge operations,
which are commonly found in the literature.
Let $v\in V$ be a vertex and 
$e = \{u, v\} \in E$ be an edge of $G$. 

\begin{itemize}
 \ditem Vertex deletion: $G-v$ denotes the graph obtained from $G$ by removal of $v$
and all edges incident to $v$. 
 \ditem Vertex contraction:  $G/v$ denotes the graph obtained from $G$ by the removal of $v$ and 
the addition of edges between any pair of non-adjacent neighbors of $v$. 
 \ditem Vertex extraction: $G-N[v]$ denotes the graph $G-N_G[v]$ obtained by deleting all of
the vertices in the closed neighborhood of $v$ and the edges incident to them.
\ditem Vertex appending: $G+\{v,\cdot\}$ denotes the graph 
$(V\cup\{v'\}, E\cup \{v,v'\})$ obtained from $G$ by adding a new vertex $v'$ and 
 an edge $\{v,v'\}$ to $G$. 
 \ditem Edge deletion: $G-e$ denotes the graph obtained from $G$ by simply removing $e$. 
 \ditem Edge contraction: $G/e$ denotes the graph obtained from $G$ by removing $e$ and 
unifying the end-points of $e$. 
\ditem Edge extraction: $G\dagger e$ denotes the graph $G-u-v$.
 Note that this operation removes all the edges incident to $e$ and $e$ itself. 
\end{itemize}

For the Sections~\ref{sec:articulations} and
\ref{sec:arbitr-separ-vert} we introduce the notion of a splitting of
a graph.
\begin{Definition}
  Let $G=(V,E)$ be a graph and $G^1=(V^1,E^1)$, $G^2=(V^2,E^2)$
  subgraphs of $G$ with $V=V^1\cup V^2$, $X=V^1\cap V^2$, $E=E^1\cup
  E^2$ and $E^1\cap E^2=\emptyset$, then we say that $(G^1,G^2,X)$ is
  a \emph{splitting of $G$} at the \emph{separating vertex set} $X$. In case of $X=\{v\}$
  the vertex $v$ is called an \emph{articulation} of $G$.
\end{Definition}
Observe that our definition of a separating vertex set and an
articulation is more general than the usual one, as the number of
components of the graph $G$ does not necessarily increase, when the
vertex set $X$ is removed from $G$.

Other notation will be introduced as needed.

\section{A recurrence relation for arbitrary graphs}\label{se:arb}

For a graph $G$ and a vertex $u$ of $G$, let $\p{u}(G,x)$ 
be the generating function counting those dominating sets 
for $G-N[u]$ which additionally dominate the vertices of $N(u)$ in $G$.
Equivalently, $\p{u}(G,x)$ is the polynomial counting the dominating sets of $G-u$ 
which do not contain any vertex of $N(u)$. 
Note $\p{u}(G,x) = D_{u\notin N_G[W]}(G-u,x)$ from Definition \ref{def:gen-dom}.
Recall that $D_{u\in W}(G,x)$ denotes 
the generating function of the dominating sets $W$ of $G$ which contain $u$.
$D_{u\notin W}(G,x)$ is defined analogously. 

The following theorem will be useful both to prove recurrence relations for arbitrary graphs
here and in Section \ref{se:der}, 
and to give recurrence steps for special cases in Section \ref{se:special}. This result 
and its two corollaries also appear in \cite{ar:Alikhanibook} but were proved independently.
\begin{thm}  \label{t:red}
Let $G = (V, E)$ be a graph.
For any vertex $u$ in $G$ we have
 \begin{align}
\label{eq:D} D(G,x) &= xD(G/u,x) + D(G-u,x) + xD(G-N[u],x) - (1+x) \p{u}(G,x)\,.
\end{align}
Furthermore, we have
\begin{align}
\label{eq:u_notin_W} D_{u \notin W}(G, x) &= D(G-u, x) - p_u(G, x)\\
\label{eq:u_in_W} D_{u \in W}(G, x) &= x D(G-N[u], x) + x (D(G/u, x) - p_u(G, x))\,.
\end{align}
\end{thm}

\begin{proof}
  Any dominating set $W$ of $G-u$ is a dominating set of $G$, unless no
neighbor of $u$ in $G$ is in $W$. The dominating sets
of $G-u$ including neither $u$ nor any of its neighbors are
enumerated by $p_u(G,x)$. Therefore we have Equation (\ref{eq:u_notin_W}). 

The dominating sets of $G$ which contain $u$ but not any of its neighbors are
counted by $x D(G-N[u], x)$. 
For every dominating set $W$ counted by
$D(G/u, x)$, $W \cup \{u\}$ is a dominating set of $G$.
The polynomial $p_u(G, u)$ counts the dominating sets of $G/u$ which do not contain any neighbor of
$u$ in $G$.
The dominating sets of $G$ which contain $u$ and at least one of its neighbors are
counted by $x(D(G/u, x) - p_u(G, u))$, and therefore we have Equation (\ref{eq:u_in_W}).

Finally, Equation (\ref{eq:D}) follows from
\[D(G,x) = D_{u\notin W}(G,x) + D_{u\in W}(G,x)\,.\qedhere
\]
\end{proof}

For every edge $\{u,v\}\in E(G)$, let $p_{u,v}(G,x)=D_{u\notin N_G[W-\{v\}], v\in W}(G-u,x)$
%
be the  generating function counting the dominating sets $W$
in $G-u$ which contain exactly one neighbor of $u$ and
that neighbor is $v$.
\begin{lem}
\label{prop:p_w} 
Let $G = (V, E)$ be a graph and
let $e=\{u,v\}\in E$. Then
\[
\p{u}(G-e,x)=p_{u,v}(G,x)+\p{u}(G,x)\,.
\]
\end{lem}
\begin{proof}
We can think of $\p{u}(G-e,x)$ as follows: $\p{u}(G-e,x)$
counts the dominating sets in $G-u$ which do not contain any
member of $N_{G}(u)$, except possibly $v$. $\p{u}(G,x)$
counts the dominating sets of $G-u$ which do not contain any
member of $N_{G}(u)$. Hence, $\p{u}(G-e,x)-\p{u}(G,x)$
counts the dominating sets $T$ in $G-w$ which contain $v$
but do not contain any other member $N_{G}(u)$.
\end{proof}
Using the previous lemma and theorem we can prove that the domination polynomial satisfies a 
recurrence relation for arbitrary graphs which is based on the edge and vertex elimination operations. 
The recurrence uses composite operations, e.g. $G-e/u$, which stands for $\left(G-e\right)/u$. 
\begin{thm}\label{th:arbitrary_rec}
Let $G = (V, E)$ be a graph.
For every edge $e=\{u,v\}\in E$, 
\begin{eqnarray*}
D(G,x) & = & D(G-e,x)+\frac{x}{x-1}\Bigg[D(G-e/u,x)+D(G-e/v,x)\\
 &  & -D(G/u,x)-D(G/v,x)-D(G-N[u],x)-D(G-N[v],x)\\
&  & +D(G-e-N[u],x)+D(G-e-N[v],x)\Bigg]\,. \\
\end{eqnarray*}
\end{thm}
\begin{proof}
Let $e=\{v,u\}\in E$. Then $D(G,x)-D(G-e,x)$ counts the dominating
sets $T$ of $G$ which are not dominating sets on $G-e$. I.e.,
$T$ is counted by $D(G,x)-D(G-e,x)$ iff one can choose $a,b\in\{u,v\}$,
$a\not=b$, such that $a\in T$, $b\notin T$ and no neighbor of $b$
in $G$, besides $a$, belongs to $T$. That means $D(G,x)-D(G-e,x)=p_{u,v}(G,x)+p_{v,u}(G,x)$. 

By Lemma \ref{prop:p_w}, $p_{u,v}(G,x)=\p{u}(G-e,x)-\p{u}(G,x)$
and $p_{v,u}(G,x)=\p{v}(G-e,x)-\p{v}(G,x)$. Thus, 
\[
D(G,x)=D(G-e,x)+\p{u}(G-e,x)+\p{v}(G-e,x)-\p{u}(G,x)-\p{v}(G,x)\,.
\]

Using Theorem \ref{t:red},
\[
D(G,x)  = D(G-e,x)+ \frac{1}{1+x}  f(G,x,e)
\]
where 
\begin{align*}
 f(G,x,e) & =    xD(G-e/u,x)+D(G-e-u,x)+xD(G-e-N[u],x)-D(G-e,x)\\
 &   +xD(G-e/v,x)+D(G-e-v,x)+xD(G-e-N[v],x)-D(G-e,x)\\
 &   -xD(G/u,x)-D(G-u,x)-xD(G-N[u],x)+D(G,x)\\
 &    -xD(G/v,x)-D(G-v,x)-xD(G-N[v],x)+D(G,x)\,.
\end{align*}
The theorem follows using that $D(G-e-u,x)=D(G-u,x)$ and $D(G-e-v,x)=D(G-v,x)$
by rearranging the terms.
\end{proof}

On the other hand, $D(G,x)$ does not satisfy any linear recurrence relation with 
the four vertex operations of deletion, contraction, extraction and neighborhood-contraction, the latter defined as follows:
we denote by $G/N(v)$ the graph obtained from 
$G$ by removing all of the vertices of $N_G(v)$ and adding edges 
so that $v$ is made adjacent to every neighbor of $N_G(v)$ in $G$.
Note $v$ is not removed from the graph.
\begin{thm} \label{th:non-exist_v}
There do not exist rational functions $a,b,c,d\in \mathbb{R}(x)$ such  that 
for every graph $G$ and vertex $v$ of $G$ it holds that
\begin{eqnarray}
D(G,x)&=&a D(G-v,x)+b D(G/v,x)+c D(G-N[v],x)+d D(G/N(v),x)\,. \label{eq:rec}
\end{eqnarray}
\end{thm}
\begin{proof}
Assume there exist $a,b,c,d$ such that Equation (\ref{eq:rec}) holds.
We will show that there exist five pairs of graphs and vertices $(G_{i},v_{i})$
such that $v_i\in V(G_i)$ for which the system of equations with indeterminates
$a,b,c,d$, 
\begin{eqnarray*}
D(G_{1},x) & = & a D(G_{1}-v_{1},x)+b D(G_{1}/v_{1},x)+cD(G_{1}-N[v_{1}],x)+dD(G_{1}/N(v_{1}),x)\\
 & \vdots\\
D(G_{5},x) & = & a D(G_{5}-v_{5},x)+b D(G_{5}/v_{5},x)+cD(G_{5}-N[v_{5}],x)+dD(G_{5}/N(v_{5}),x)
\end{eqnarray*}
has no solution. The graphs are: $G_{1}=K_1$, $G_{2}=K_2$, $G_{3}=K_3$,
and $G_{4}=G_5=P_{3}$. For $G_{1}$, $G_{2}$ and $G_{3}$ all of the
vertices are symmetric. We choose $v_{4}$ as one of the two end-points
of $P_{3}$, and $v_{5}$ as the central vertex in $P_{3}$. We have:
\begin{center}
\begin{tabular}{|c||c|c|c|c|}
\hline 
$G_{i}\ \mbox{(}v_i\mbox{)}$ & $G_{i}-v_{i}$ & $G_{i}/v_{i}$ & $G_{i}-N[v_{i}]$ & $G_{i}/N(v_{i})$\tabularnewline
\hline 
\hline 
$K_1$ & $\emptyset$ & $\emptyset$ & $\emptyset$ & $K_1$\tabularnewline
\hline 
$K_2$ & $K_1$ & $K_1$ & $\emptyset$ & $K_1$\tabularnewline
\hline 
$K_{3}$ & $K_2$ & $K_2$ & $\emptyset$ & $K_1$\tabularnewline
\hline 
$P_{3}\ $($v_3$ is an end-point) & $K_2$ & $K_2$ & $K_1$ & $K_2$\tabularnewline
\hline 
$P_{3}\ $($v_4$ is the central vertex) & $K_1\cup K_1$ & $K_2$ & $\emptyset$ & $K_1$\tabularnewline
\hline 
\end{tabular}
\end{center}

and 
\begin{center}
\begin{tabular}{|c||c|c|c|c|}
\hline 
$D(G_{i},x)$ & $D(G_{i}-v_{i},x)$ & $D(G_{i}/v_{i},x)$ & $D(G_{i}-N[v_{i}],x)$ & $D(G_{i}/N(v_{i}),x)$\tabularnewline
\hline 
\hline 
$x$ & $1$ & $1$ & $1$ & $x$\tabularnewline
\hline 
$2x+x^{2}$ & $x$ & $x$ & $1$ & $x$\tabularnewline
\hline 
$3x+3x^{2}+x^{3}$ & $2x+x^{2}$ & $2x+x^{2}$ & $1$ & $x$\tabularnewline
\hline 
$x+3x^{2}+x^{3}$ & $2x+x^{2}$ & $2x+x^{2}$ & $x$ & $2x+x^{2}$\tabularnewline
\hline 
$x+3x^{2}+x^{3}$ & $x^{2}$ & $2x+x^{2}$ & $1$ & $x$\tabularnewline
\hline 
\end{tabular}
\end{center}
but the corresponding system of equations has no solution.
\end{proof}
The proof technique here is general in the sense that it can be extended to other 
vertex elimination operations. Note, however, that for the specific operations we used here,
taking the three graphs $G_1$, $G_2$ and $G_3$ only suffices. 

Similarly, $D(G,x)$ does not satisfy any linear recurrence relation with 
the three edge operations of deletion, contraction and extraction.

\begin{thm}\label{th:non-exist_e}
There do not exist rational functions $a,b,c\in \mathbb{R}(x)$ such that
for every graph $G$ and edge $e$ of $G$ it holds that
\begin{eqnarray}
D(G,x)&=&a D(G-e,x)+b D(G/e,x)+c D(G\dagger e,x)\,.\notag
\end{eqnarray}
\end{thm}
The proof of Theorem \ref{th:non-exist_e} is similar to that of
Theorem \ref{th:non-exist_v}
with $G_1 = K_2$, $G_2 = P_3$, $G_3=K_3$ and $G_4=P_4$, and where the edge associated with $G_4$ is the middle edge.

\section{Simple recurrence steps for special cases}\label{se:special}

The domination polynomial of a graph $G$ can be computed entirely
using the recurrence relation in Theorem \ref{th:arbitrary_rec} and the multiplicativity of $D(G,x)$ with respect to disjoint union. 
However, in many cases, a simpler recurrence step will suffice. 

\begin{proposition}\label{c:nbr}
Let $G=(V,E)$ be a graph.
If $u,v \in V$ and $N[v] \subseteq N[u]$ then
\begin{gather}\label{eq:c:nbr:1}
D(G,x) = x D(G/u,x) + D(G-u,x) + x D(G-N[u],x) 
\end{gather}
and 
\begin{gather}\label{eq:c:nbr:2}
D(G,x) = (1+x) D(G-u,x) + D_{u \in W}(G-\left(N[v]\backslash\{u\} \right),x) \,.
\end{gather}
\end{proposition}

\begin{proof}
For Equation (\ref{eq:c:nbr:1}) we use Theorem \ref{t:red}.
In this case $\p{u}(G,x)=0$ since if neither $u$ nor a vertex adjacent to 
it is in the dominating set, then $v$ (and its neighbors except $u$) can not be dominated.


For Equation (\ref{eq:c:nbr:2}) we partition the dominating sets of $G$ into two types, counted by the two terms of Equation (\ref{eq:c:nbr:2}).
$(1+x) D(G-u,x)$ counts the dominating sets of $G$ which contain at least one vertex from $N_G[v]\backslash \{u\}$.
$D_{u \in W}(G-\left(N[v]\backslash\{u\} \right),x)$ counts the dominating sets of $G$ which contain no element of $N_G[v]$ except for $u$. 
\end{proof}

\begin{proposition}\label{c:not}
Let $G=(V,E)$ be a graph.
If $u,w \in V$ and $N(w) = N(u)$ then
\begin{eqnarray*}
D(G,x) = x D(G/u,x) + D(G-u,x) - x D(G-N[u]-w,x) \,.
\end{eqnarray*}
\end{proposition}

\begin{proof}
In this case $w$ is
of degree 0 in $G-N[u]$ and hence $w$ 
must exist in any dominating set for $G-N[u]$.
Since $N(w)=N(u)$ we know that $N(u)$ is dominated by every dominating
set for $G-N[u]$.
Thus we have
 $\p{u}(G,x) = D(G-N[u],x)$.
It remains to notice that $D(G-N[u],x)=xD(G-N[u]-w,x)$ since $w$ is an isolated vertex in $G-N[u]$. 
The result follows from Theorem \ref{t:red}.
\end{proof}

 Here we generalize Proposition \ref{c:not}:
 \begin{proposition}
Let $G=(V,E)$ be a graph.
 If $u,w \in V$ and $N(w) \subseteq N(u)$ then,
 \begin{eqnarray}
 D(G,x) &=& x D(G/u,x) + D(G-u,x) + x D(G-N[w],x) \notag \\
& & - x^2 D(G-N[w]/u,x) -x D(G-N[w]-u,x)\,.  \label{e:wnot}
 \end{eqnarray}
  
 \end{proposition}

 \begin{proof}
 
 Applying Theorem \ref{t:red} on $G-N[w]$ and $G$, both with the vertex $u$, and rearranging the terms we get
 \begin{gather*}
 D(G-N[w],x) - x D(G-N[w]/u,x) -D(G-N[w]-u,x) = \\ 
 x D(G-N[w]-N[u],x)  - (1+x) \p{u}(G-N[w],x) 
\end{gather*}
and 
\begin{gather*}
 D(G,x) - x D(G/u,x) - D(G-u,x) = x D(G-N[u],x) - (1+x) \p{u}(G,x)\,.
 \end{gather*}

Using that $G-N[u] = \{w\} \cup (G-N[w]-N[u])$ we have 
$D(G-N[u],x) = x D( G-N[w]-N[u],x )$ and $\p{u}(G,x)=x\p{u}(G-N[w],x)$, and
the proposition follows. 
 \end{proof}

\begin{example}[Vertices of degree 1]
Let $G=(V,E)$ be a graph.
Let $v$ be a vertex of degree $1$ in a graph $G$ and let $u$ be its neighbor.
Then $N[v]= \{u,v\} \subseteq N[u] $ and so, by Proposition \ref{c:nbr}:
\begin{eqnarray*}
D(G,x) &=& x D(G/u,x) + D(G-u,x)  + x D(G-N[u],x)  \\
 &=& x D(G/u,x) + x D(G-u-v,x)  + x D(G-N[u],x)  \\
 &=& x \left( D(G/u,x) + D(G-u-v,x)  + D(G-N[u],x)  \right)\,.
\end{eqnarray*}
\end{example}

\begin{example}[Vertices of degree $|V(G)| -2$]
Let $G=(V,E)$ be a graph.
Let $u$ be such a vertex adjacent to all vertices in $G$ except for $w$.
If there exists a vertex $v \in N(u)$ such that $\{v,w\} \not \in E$ then
$N[v] \subseteq N[u]$ and so Proposition \ref{c:nbr} will apply:
\begin{eqnarray*}
D(G,x) &=& x D(G/u,x) + D(G-u,x) + x D(G-N[u],x)  \\
 &=& x D(G/u,x) + D(G-u,x) + x^2\,.
\end{eqnarray*}

Otherwise, $N(w)=N(u)$ and hence by Proposition \ref{c:not} 
\begin{eqnarray*}
D(G,x) &=& x D(G/u,x) + D(G-u,x) - D(G-N[u],x)  \\
 &=& x D(G/u,x) + D(G-u,x) - x \\
 &=& x ((1+x)^{|V|-1} -1 ) + D(G-u,x) - x \\
 &=& x(1+x) ((1+x)^{|V|-2} - 1 ) + D(G-u,x) \,.
\end{eqnarray*}
\end{example}

\subsection{Removing triangles}

Theorem \ref{t:red} can be used to give a recurrence relation which removes triangles. 
We define a new operation on edges incident to a vertex $u$: 
we denote by $G\triangleOp u$ the graph obtained from $G$ by 
the removal of all edges between any pair of neighbors of $u$. Note $u$ is not removed from the graph.
This recurrence relation is useful on graphs which have many triangles. 

\begin{proposition}\label{th:clearing}
Let $G=(V,E)$ be a graph and $u\in V$. Then
\[
 D(G,x) =  D(G-u,x) + D(G \triangleOp u,x) - D(G\triangleOp u -u,x)  \,.
\] 
\end{proposition}
\begin{proof}
First note that, since the operation $\triangleOp u$  only removes the edges between vertices in $N(u)$, 
these relations will hold:
\[
(G \triangleOp u)/u \cong G{/ u}   ~,~~~%
\p{u}(G,x)= \p{u}(G\triangleOp u,x)   ~,~~~
(G\triangleOp u) - N[u] \cong (G-N[u])\,.
\]
Using these relations with Theorem \ref{t:red}, we have
\begin{eqnarray*}
 D(G,x) -  D(G-u,x) &=& 
x D(G/u,x) + x D(G-N[u],x) - (1+x) \p{u}(G,x) \\
 &=&
x D((G \triangleOp u)/u,x) + x D((G\triangleOp u)-N[u],x) \\ 
&&- (1+x) \p{u}(G\triangleOp u,x) \,.
\end{eqnarray*}
Now we apply Theorem \ref{t:red} to $G\triangleOp u$ and get that the right-hand side of the latter equation
equals
\[
D(G \triangleOp u,x) - D((G \triangleOp u)-u,x)\,,
\]
and the proposition follows.
\end{proof}

\subsection{Induced \texorpdfstring{$5$}{5}-paths}
In \cite{pathof2} the following recurrence formula was proved in terms 
of edge contractions:
\begin{theorem}
Suppose $G$ is a graph which contains five vertices $u,v,w,y,z$ that form a path in that order such that the degrees of $v$, $w$ and $y$ are 2.
Then
\begin{equation}\label{t:path}
D(G,x) = x  \left( 
D(G/w,x) +
D(G/v/w,x) +
D(G/v/w/y,x) 
\right)\,.
\end{equation}
\end{theorem}
Note that the order of contracting $v$, $w$ and $y$ in $D(G/v/w/y,x)$ does not matter. The same is true 
for deleting $v$, $w$ and $y$ in $D(G-v-w-y,x)$. 

We can verify this result using Theorem \ref{t:red} as follows:

\begin{proof}
Suppose the five vertices in the induced 
path are $u$, $v$, $w$, $y$ and $z$ in order along the path.
We apply Theorem \ref{t:red} to the central vertex, $w$:
\[
D(G,x) = x D(G/w,x) + D(G-w,x) + x D(G-v-w-y,x) - (1+x) \p{w}(G,x)\,.
\]

Note that the first term is the same as in Equation (\ref{t:path}), 
so it remains to prove that
\begin{equation}\label{e:path}
 x D(G-v-w-y,x) - (1+x) \p{w}(G,x) 
 + D(G-w,x)  =
x D(G/v/w,x) + x D(G/v/w/y,x) \,.
\end{equation}
Let $H:= G-N[w]$, and let $S$ be a dominating set for $H$.
We wish to extend $S$ to a dominating set for each of the graphs in 
Equation (\ref{e:path}) by considering whether or not $v$ and/or $y$ must
or may be added to $S$.
For $S$ to dominate $N(w)$ it
must include both $u$ and $z$. We break the proof here
into 3 cases, dependent on how many of $u$ and $z$ are in $S$.


When $u$, say, 
is in $S$ it will dominate $v$ and so $v$ can either be in $S$ or out
of it, giving a factor of $(1+x)$ to multiply the domination polynomial 
of $H$ by. 
If $u$ is not in $S$ then $v$ must 
be in $S$ in order for $S$ to be a dominating set, giving a factor of $x$
for $H$.
We tabulate the respective contributions for vertices $v$ and $y$ in
the different graphs, substituting
 $q(x):=x D(G-v-w-y,x) - (1+x) \p{w}(G,x)$:

\begin{table}[h]
\begin{center}
\begin{tabular}{|c||c|c|c|c|}\hline
$|S \cap \{w,y\}|$ & 
 $q(x)$ &
 $D(G-w,x)$ &
$x D(G/v/w,x)$ & 
$x D(G/v/w/y,x) $ \tabularnewline
 \hline \hline 
2 & -1 & $(1+x)^2$ & $x(1+x)$ & $x$ \tabularnewline
1 & x & $x(1+x)$ & $x(1+x)$ & $x$ \tabularnewline
0 & x & $x^2$ & $x^2$ & $x$ \tabularnewline
\hline
\end{tabular}
\end{center}
\caption{Table of contributions from vertices $v$ and $y$}
\end{table}

For each of these rows we can see that Equation (\ref{e:path}) is satisfied
by adding both pairs of columns.
Since all of the possibilities for $S$
fall into one of these three cases, the proof is complete.
\end{proof}

\subsection{Irrelevant edges}
The easiest recurrence relation one might think of is to remove an edge and
to compute the domination polynomial of the graph arising instead of
the one for the original graph. Indeed, for the domination polynomial
of a graph there might be such irrelevant edges, that 
can be deleted without changing the value of the domination
polynomial at all. We characterize edges possessing this property.

\begin{Definition}
Let $G = (V,E)$ be a graph. An \emph{irrelevant edge} is an edge $e \in E$ of $G$, such that
\begin{align}
  D(G, x) = D(G-e, x).
\end{align}
\end{Definition}

\begin{Definition}
  Let $G = (V, E)$ be a graph. A vertex $v \in V$ of $G$ is
  \emph{domination-covered}, if every dominating set of $G-v$
  includes at least one vertex adjacent to $v$ in $G$.
\end{Definition}
In other words, each dominating set of $G-v$ is a dominating set
of $G$ and the dominating sets of $G-v$ are exactly those
dominating sets of $G$ not including the vertex $v$.

\begin{Theorem}
  Let $G = (V, E)$ be a graph. A vertex $v \in V$ of $G$ is
  domination-covered if and only if there is a vertex $u \in N_G[v]$
  such that $N_G[u] \subseteq N_G[v]$.
\end{Theorem}
In case the condition of the theorem above is satisfied, we say that
the vertex $v$ is domination-covered by the vertex $u$.

\begin{proof}
  First we prove that, if there is such a vertex $u$, then $v$ is
  domination-covered. To dominate $u$, either $u$ or a vertex
  adjacent to $u$ must be in each dominating set of
  $G-v$. Because every vertex adjacent to $u$ is also adjacent
  to $v$ in $G$, the statement follows.

  Secondly we prove that, if $v$ is domination-covered, there is such
  a vertex $u$. Assume that $v$ is domination-covered, but there
  is no such vertex, that means each vertex adjacent to $v$ in $G$
  has at least one neighbor not adjacent to $v$ in $G$. Then the
  set $V \setminus N[v]$ is a dominating set of $G-v$, but not
  of $G$, which contradicts the assumption.
\end{proof}

\begin{Theorem}
  Let $G = (V, E)$ be a graph. An edge $e = \{u, v\} \in E$ is an
  irrelevant edge in $G$ if and only if $u$ and $v$ are
  domination-covered in $G-e$.
\end{Theorem}

\begin{proof}
  We start with the graph $G-e$ and show that no additional
  dominating set arises by inserting the edge $e$ if and only if
  $u$ and $v$ are domination-covered in $G-e$.

  First we prove, that if both vertices are domination-covered, then
  $e$ is an irrelevant edge. Assume $e$ is not irrelevant, that
  means there is a vertex subset $W$ which is a dominating set in
  $G$ but not in $G-e$. Consequently the edge $e$ must be
  ``used'', meaning that w.l.o.g. $u \in W$ and $N_{G-e}[v]
  \cap W = \emptyset$. But then $v$ is not domination-covered in
  $G-e$, which contradicts the assumption.

  Second we prove, that if $e$ is an irrelevant edge, then both
  vertices are domination-covered. Assume at least one vertex is not
  domination-covered in $G-e$, say $u$. Then $V \setminus
  N_{G-e}[v]$ is a dominating set in $G$ but not in $G-e$,
  which contradicts the assumption.
\end{proof}

The theorem above can be applied to the calculation of the domination
polynomial of corona graphs $G \circ H$.

For two graphs $G = (V, E)$ and $H$, the \emph{corona} $G \circ H$ is
the graph arising from the disjoint union of $G$ with $\abs{V}$ copies
of $H$, named $H_1,\ldots, H_{|V|}$, by adding edges between the $i$th vertex of $G$
and all vertices of $H_i$ for $1\le i\le |V|$.

\begin{Theorem}
Let $G = (V, E)$ and $H = (W, F)$ be non-empty graphs. Then
\begin{align}
D(G \circ H, x) = [D(K_1 \circ H, x)]^{\abs{V}} = [x(1+x)^{\abs{W}} + D(H, x)]^{\abs{V}}\,.
\end{align}
\end{Theorem}

\begin{proof}
  In the corona of two non-empty graphs, each vertex $u \in V$ of $G$
  is adjacent to all vertices of the corresponding copy of $H$,
  therefore the vertex $u$ is domination-covered (by each vertex of
  the according copy of $H$). Thus, we can delete all edges in $E$
  (from the original graph $G$) in the corona and the arising
  graph is the disjoint union of $\abs{V}$ copies of the corona $K_1
  \circ H$, which proves the first identity. In each of these coronas
  a dominating vertex set either includes $u$ (the vertex originally
  in $G$) and an arbitrary subset of the $\abs{W}$ vertices from the
  copy of $H$. Or, it does not include $u$ and equals exactly a
  dominating vertex set of (the copy of) $H$, which proves the second
  identity.
\end{proof}

This result generalizes the statement for a special corona graph $H =
K_1$ as given in \cite[Theorem 1]{Akbari2010}.

\section{Articulations\label{sec:articulations}\label{se:art}}
For many graph polynomials, it is possible to calculate the
polynomial of a graph with an articulation by calculating
the polynomial of some graphs related to those graphs
arising from separating the graph at the articulation. Here we will prove
that this idea is also workable for the domination polynomial.

Throughout this section we will be
considering three polynomials from Definition \ref{def:gen-dom}: 
\(D_{v \in W}(G, x)\),
\(D_{v \notin W}(G, x)\) and \(D_{v \notin N_G[W]}(G-v,x) = p_{v}(G,x)\).
The first
and the second polynomials count exactly those dominating vertex sets of
\(G\) including \(v\) and not including \(v\), respectively. The third
polynomial counts the vertex subsets of $G-N[v]$ which dominate the vertices
of $V\setminus\{v\}$ in $G$.
Alternatively, this polynomial is also $p_v (G,x)$
(from Section \ref{se:arb}) and
counts all dominating vertex sets
of \(G-v\) that do not include any neighbor of \(v\) (in \(G\)).

\begin{Definition}\label{de:art}
  Let $G=(V,E)$ be a graph and $v\in V$. We define the following two vectors
  \begin{align*}
    \uvec_v(G) &= 
    \begin{pmatrix}
      D_{v \notin W}(G, x) \\
      D_{v \notin N_G[W]}(G-v, x) \\
      \frac{1}{x} D_{v \in W}(G, x)
    \end{pmatrix}\,,
    &
    \dvec_v(G)&=
    \begin{pmatrix}
      D(G-v,x) \\
      D(G,x) \\
      D(G+\{v,\cdot\},x)
    \end{pmatrix}
  \end{align*}  
  and the matrices $\Pmat$ and $\Qmat$ by
  \begin{align*}
    \Pmat&=
    \begin{pmatrix}
      1 & 0 & 0 \\
      0 & 1 & 1 \\
      x & 0 & 1+x
    \end{pmatrix}\,,
    &
    \Qmat &= 
    \begin{pmatrix}
      1 & 1 & 0 \\
      1 & 0 & 0 \\
      0 & 0 & x
    \end{pmatrix}\,.
  \end{align*}
\end{Definition}

The domination polynomial of a graph \(G\) with an articulation \(v\)
can be computed from the vectors defined above for the graphs arising
from the splitting at \(v\) as shown below:

\begin{Theorem}\label{thm:articulation}
Let $G=(V,E)$ be a graph with splitting $(G^1,G^2,\{v\})$ and articulation $v\in V$. 
Then
\begin{align}
D(G, x) &= \uvec_v(G^1)^T \Qmat \uvec_v(G^2)\,.\label{eq:Q12}
\end{align}
\end{Theorem}

\begin{proof}
Expanding equation (\ref{eq:Q12}) we see that matrix $\Qmat$ combines
   all of the different ways that dominating sets can arise in $G$.
   For instance, 
the first two rows of $\Qmat$ give the expression
\[
 D_{v \notin W}(G^1, x) \left( D_{v \notin W}(G^2, x) + 
D_{v \notin N_{G^2}[W]}(G^2-v, x)\right) + 
D_{v \notin W}(G^1, x) D_{v \notin N_{G^2}[W]}(G^2-v, x) 
\]
which corresponds to saying that 
if $v$ is not in a dominating set for $G$ then it must either
be dominated by a vertex in both $G^1$ and $G^2$ or only one of them.
If \(v\) is in $W$ in \(G^1\) it must
  also be in $W$ in \(G^2\) so only the $(3,3)$ entry of $\Qmat$
will be involved. However \(v\) will be
  counted twice and so we will have to multiply by
  \(x\) to get the expression 
\[
 \frac{1}{x} D_{v \in W}(G^1, x) D_{v \in W}(G^2, x) 
\]
as required.
\end{proof}

\begin{Theorem}\label{thm:utodvector}
Let \(G = (V, E)\) be a graph and \(v \in V\). Then
\begin{align*}
  \dvec_v(G)&=\Pmat\Qmat\uvec_v(G)\,.
\end{align*}
\end{Theorem}
\begin{proof}
  The first row of the stated matrix equation gives
  \begin{align*}
    D(G-v,x)&=D_{v\notin W}(G,x)+D_{v\notin N_G[W]}(G,x)\,,
  \end{align*}
  which is a valid equation, as we are counting all dominating vertex
  sets of the graph $G-v$. 
For the remaining rows we will utilise Theorem~\ref{thm:articulation}.
First, let us consider the splitting 
  $(G^1,G^2,\{v\})$ of $G$ with $G^1=(\{v\},\emptyset)$
  and $G^2=G=(V,E)$ and observe that 
  \begin{align*}
    \uvec_v(G^1)^T=&
    \begin{pmatrix}
      0, & 1, & 1
    \end{pmatrix}\,,
  \end{align*}
  which equals the second row of the matrix $\Pmat$ as required.
  
 Now choose $G^1=(\{v,v'\}, \{\{v,v'\}\})$ with $v'\notin V$
  and $G^2=G=(V,E)$ and, similarly,
  \begin{align*}
    \uvec_v(G^1)^T&=
    \begin{pmatrix}
      x, & 0, & 1+x
    \end{pmatrix}\,,
  \end{align*}
  Theorem~\ref{thm:articulation} again proves that 
the final row of the stated matrix equation is valid.
\end{proof}

\begin{Lemma}\label{lem:dmatvFactorisation}
Under the requirements of Definition \ref{de:art}
  we obtain the matrix
  \begin{align*}
   \Dmat&=\Pmat\Qmat\Pmat^{T} =
    \begin{pmatrix}
      1 & 1 & x \\
      1 & x & 2x+x^2\\
      x & 2x+x^2 & x+3x^2+x^3
    \end{pmatrix}\,,
  \end{align*}
  which is expressible as domination polynomials of paths
  \begin{align*}
        \Dmat&=
    \begin{pmatrix}
      D(P_0,x) & D(P_0,x) & D(P_1,x) \\
      D(P_0,x) & D(P_1,x) & D(P_2,x)\\
      D(P_1,x) & D(P_2,x) & D(P_3,x)
    \end{pmatrix}\,.
  \end{align*}
\end{Lemma}

\begin{Theorem}
  Let $G=(V,E)$ be a graph with splitting $(G^1,G^2,\{v\})$ and
  articulation $v\in V$.  Then the domination polynomial of $G$ can be expressed
as follows:
  \begin{align*}
    D(G,x)&=\dvec_v(G^1)^T\Dmat^{-1}\dvec_v(G^2)\,.
  \end{align*}
\end{Theorem}
\begin{proof}
  By Theorem~\ref{thm:articulation} we have
  \begin{align}\label{eq:proofFirst}
    D(G,x)&=\uvec_v(G^1)^T\Qmat\uvec_v(G^2)
  \end{align}
  and Theorem~\ref{thm:utodvector} gives
  \begin{align}\label{eq:proofSecond}
    \uvec_v(G^i)&=\Qmat^{-1}\Pmat^{-1}\dvec_v(G^i) 
  \end{align}
  for $i=1,2$. Substitution of Equation~\eqref{eq:proofSecond} into
  Equation~\eqref{eq:proofFirst} gives the result.
\end{proof}

Another recurrence relation for graphs with an articulation can be obtained
using Theorem \ref{t:red}. 
This recurrence relation uses only the graphs $G^1$ and $G^2$ and their subgraphs. 
\begin{theorem}
Let $G=(V,E)$ be a graph with splitting $(G^1,G^2,\{v\})$ and articulation $v\in V$. Then
\begin{eqnarray*}
D(G,x)  &=&  x D(G/v,x) 
+ D(G^{1}-v,x)\cdot D(G^{2}-v,x)\\
& &  + x D(G^{1}-N[v],x)\cdot D(G^{2}-N[v],x) \\
 &&- 
\frac{1}{1+x}
\left[\left(
x D(G^{1}/v,x) + D(G^{1}-v,x)  + x D(G^1-N[v],x) - D(G^1,x) \right) 
\right.
\\ &&
\left. \cdot \left(
 x D(G^{2}/v,x) + D(G^1-v,x)  + x D(G^2-N[v],x) - D(G^2,x)
\right) \right]\,.
\end{eqnarray*}

\end{theorem}
\begin{proof}
By Theorem \ref{t:red} we have, for the graphs $G$, $G^1$ and $G^2$ and vertex $v$: 
\begin{eqnarray}\label{1conn}
(1+x) \p{v}(G,x) &=&  x D(G/v,x) + D(G-v,x)  + x D(G-N[v],x) - D(G,x)  \\
(1+x) \p{v}(G^1,x) &=&  x D(G^1/v,x) + D(G^1-v,x)  + x D(G^1- N[v],x) - D(G^1,x)  \notag \\  
(1+x) \p{v}(G^2,x) &=&  x D(G^2/v,x) + D(G^2-v,x)  + x D(G^2- N[v],x) - D(G^2,x)\,.  \notag
\end{eqnarray}

$\p{v}(G,x)$ counts exactly those dominating sets of $G-N[v]$ 
which dominate $N(v)$. Since $v$ is an articulation, each such set must be
a disjoint union of dominating sets 
for $G^{i}$ which each dominates $N(v) \cap G^{i} = N_{G^i}(v)$ for $i=1,2$.
 Hence
\begin{equation}\label{e:1conn}
\p{v}(G,x) = \p{v}(G^1,x) \cdot \p{v}(G^2,x)\,.
\end{equation}

The graphs $G-v$ (respectively $G-N[v]$) are the disjoint union of $G^1-v$ and $G^2-v$ (respectively $G^1-N[v]$ and $G^2-N[v]$). 
Hence, 
\begin{eqnarray} \label{e:1conn_m}
D(G-v,x) &=& D(G^1-v,x) \cdot D(G^2-v,x) \\ 
D(G-N[v],x) &=& D(G^1-N[v],x) \cdot D(G^2-N[v],x)\,. \notag
\end{eqnarray}
From Equations (\ref{1conn}), (\ref{e:1conn}) and (\ref{e:1conn_m}) we get the desired result by 
rearranging the terms. 
\end{proof}

\section{Arbitrary separating vertex sets\label{sec:arbitr-separ-vert}\label{se:sep}}
We will now extend the results in Section~\ref{sec:articulations} from
articulations to arbitrary separating vertex sets. 
Theorem \ref{thm:symmDomSplittingFormula}, 
the Splitting Formula Theorem, will then be applied to a simple separating 
set of size $2$ in Subsection \ref{subse:application}. 

\begin{Definition}
  Let $S$ be a finite set. We denote by $\Mat(S,S)$ the set of all
  square matrices with $|S|$ rows and columns, where the rows and
  columns are indexed by the elements of $S$. Similarly, we denote by
  $\Mat(S)$ the set of all vectors with $|S|$ rows indexed by the
  elements in $S$.
\end{Definition}

\begin{Definition}
  Let $S^1,S^2$ be finite sets. Given two matrices $\Umat\in\Mat(S^1,S^1)$ and $\Vmat\in\Mat(S^2,S^2)$ with
  \begin{align*}
    \Umat&=(u(x,y))_{x,y\in S^1} \\
    \Vmat&=(v(x,y))_{x,y\in S^2}
  \end{align*}
  we define the Kronecker product
  \begin{align*}
    \Umat\otimes \Vmat&=(u(x,y)v(r,s))_{(x,r)\in S^1\times S^2, (y,s)\in S^1\times S^2}
  \end{align*}
  of $\Umat$ and $\Vmat$ being an element in $\Mat(S^1\times S^2,S^1\times S^2)$.
\end{Definition}

\begin{Proposition}[see~\cite{marcus1992}]\label{prop:kroneckerInverse}
  Let $U,V$ be finite sets and $\Umat\in\Mat(U,U)$, $\Vmat\in\Mat(V,V)$ be invertible matrices.
  We have that
  \begin{align}
    (\Umat\otimes \Vmat)^{-1}&=\Umat^{-1}\otimes \Vmat^{-1}\,.
  \end{align}
\end{Proposition}

  We will denote by $2^X$ the power set of a given set $X$.
\begin{Definition}
  Let $G$ be a graph with splitting $(G^1,G^2,X)$. We define the set $R(X)$ as follows:
  \begin{align*}
    R(X)&=\{(A_X,B_X)\in 2^X\times 2^X:A_X\subseteq B_X\}\,,
  \end{align*}
\end{Definition}
  
Note that $R(X)$ has
$3^{|X|}$ elements and for all $v\in X$ the three-element set
  \begin{align*}
    R(\{v\})&=\{(A_v,B_v)\in 2^{\{v\}}\times 2^{\{v\}}:A_v\subseteq B_v\} \\
    &=\{(\varnothing,\varnothing),(\varnothing,\{v\}),(\{v\},\{v\})\}\,.
  \end{align*}

  In the following section we will use the bijection $f$ which has these
two properties
  when dealing with Kronecker products of matrices in $\Mat(R(\{v\}),R(\{v\}))$ for $v\in X$. 
  \begin{align*}
     \prod_{v\in X} R(\{v\})\mapsto R(X) ~~,~~~~
    \prod_{v\in X}(A_v,B_v)\mapsto \left(\bigcup_{v\in X} A_v, \bigcup_{v\in X} B_v\right)\,,
  \end{align*}
  Assume now that $U$ and $V$ are finite and disjoint sets.
  Under this restriction, if 
$\Amat$ and $\Bmat$ have rows and columns indexed by $R(U)$ and $R(V)$,
  then $\Amat\otimes \Bmat$ has rows and columns indexed by $R(U\cup V)$, following the above bijection.

  Finally, for every vertex $v$, the rows and columns of the 
matrices 
      \(
\Pmat_v,\Qmat_v,\Dmat_v\in\Mat(R(\{v\}),R(\{v\}))
\)
defined in an analogous way to the matrices in Section \ref{se:art}
will always be given with this specific ordering of the elements in $R(\{v\})$: 
  \begin{align*}
    (\varnothing, \{v\}), (\varnothing, \varnothing), (\{v\},\{v\}).
  \end{align*}
For example we have in the third row
  and column of matrix $\Pmat_v$, from Section~\ref{sec:articulations}, the entry $1+x$, that is
  corresponding to the row and the column indexed by $(\{v\},\{v\})$.

\begin{Definition}
  Let $(G^1,G^2,X)$ be a splitting of $G$ with $G^i=(V^i,E^i)$. We define the Kronecker products
\[
    \Pmat_X=\bigotimes_{v\in X} \Pmat_v ~,~~~   \Qmat_X=\bigotimes_{v\in X} \Qmat_v ~,~~~      \Dmat_X=\bigotimes_{v\in X} \Dmat_v\,.
\]
\end{Definition}

\begin{Proposition}\label{prop:matrixInvertible}
  The matrices $\Pmat_X$ and $\Qmat_X$ are invertible.
\end{Proposition}
\begin{proof}
  For every $v\in X$ the matrices $\Pmat_v$ and $\Qmat_v$ are
  invertible, as $\det(\Pmat_v)\ne 0$ and $\det(\Qmat_v)\ne 0$.  Furthermore we
  conclude by Proposition~\ref{prop:kroneckerInverse}
\[
    \Pmat^{-1}_X=\bigotimes_{v\in X} \Pmat_{v}^{-1} ~,~~~
    \Qmat^{-1}_X=\bigotimes_{v\in X}\Qmat_v^{-1}~,~~~
\]
  which proves the claim.
\end{proof}

\begin{Proposition}\label{prop:DmatrixFactorisation}
  We have
  \begin{align}
    \Dmat_X&=\Pmat_X\Qmat_X\Pmat^T_X.
  \end{align}
\end{Proposition}
\begin{proof}
  Observe that we have the following chain of equalities
\begin{align*}
    \Pmat_X \Qmat_X \Pmat^{T}_X&=\bigotimes_{v\in X}\Pmat_v \cdot \bigotimes_{v\in X} \Qmat_v\cdot \bigotimes_{v\in X}\Pmat_v^T & \tag{Definition}\\
    &=\bigotimes_{v\in X}(\Pmat_v\Qmat_v\Pmat_v^T) & \tag{Property of the Kronecker product}\\
    &=\bigotimes_{v\in X}\Dmat_v & \tag{Lemma~\ref{lem:dmatvFactorisation}} \\
    &=\Dmat_X\,.&   \tag*\qedhere
  \end{align*}

\end{proof}

\begin{Definition}\label{def:statePolynomialsDominationPolynomial}
  Let $G=(V,E)$ be a graph with splitting $(G^1,G^2,X)$ with $G^i=(V^i,E^i)$. 
We define 
  \begin{align*}
    D(G^i,(A,B),x)&=\sum_{W^i\subseteq V^i} [W^i\cap X=A, N^i\cap X=B][N^i\setminus
    X=V^i\setminus X] x^{|W^i\setminus X|}
  \end{align*}
  for all $(A,B)\in R(X)$ with $N^i:=N_{G^i}[W^i]$ and $i=1,2$. Furthermore
  we  define the vectors $\uvec_X(G^i)\in\Mat(R(X))$ by
  \begin{align}
    \uvec_X(G^i)&=(D(G^i,(A,B),x))_{(A,B)\in R(X)}
  \end{align}
  for $i=1,2$.
\end{Definition}

\begin{Remark}
  The reader should observe that the vector $\uvec_X(G^i)$ is the direct generalization of the 
  vector $\uvec_v(G^i)$ in Section~\ref{sec:articulations}. Indeed, 
  for an articulation $\{v\}$ as a singleton vertex separating set $X=\{v\}$,
we obtain the following three identities
  \begin{align*}
    &D(G,(\emptyset,\emptyset),x)=D_{v\not\in N_G[W]}(G-v,x)\,, \\
    &D(G,(\emptyset,\{v\}),x)=D_{v\not\in W}(G,x)\,, \\
    &D(G,(\{v\},\{v\}),x)=D_{v\in W}(G,x)/x\,.
  \end{align*}
  In other words we have $\uvec_{\{v\}}(G^i)=\uvec_v(G^i)$.
\end{Remark}

\begin{Lemma}\label{lem:glueingLemmaDomination} 
  Let $G=(V,E)$ be a graph with splitting $(G^1,G^2,X)$, with
  $G^i=(V^i,E^i)$ for $i=1,2$. We have, for all $W\subseteq V$,
  \begin{align*}
    |W|&=|W^1\setminus X|+|(W^1\cap X) \cup (W^2\cap X)|+|W^2\setminus X| 
  \end{align*}
  and
  \begin{align*}
    [N_G(W)=V]=
    [N^1\setminus X=V^1\setminus X][(N^1\cup N^2)\cap X=X][N^2\setminus X=V^2\setminus X]
  \end{align*}
  where $W^i:=W\cap V^i$ and $N^i:=N_{G^i}[W^i]$ for $i=1,2$.
\end{Lemma}
\begin{proof}
  The first claim follows directly by partitioning of $W$ into three disjoint sets
  \begin{align*}
    W&=(W^1\setminus X)\cup ((W^1\cap X)\cup (W^2\cap X)) \cup (W^2\setminus X)\,.
  \end{align*}
  Similarly we obtain 
  \begin{align*}
    N_G[W]&= (N^1\setminus X) \cup ((N^1\cap X)\cup (N^2\cap X)) \cup (N^2\setminus X)\,,
  \end{align*}
  which obviously implies the second claim.
\end{proof}

We will now give the generalization of 
Theorem~\ref{thm:articulation} to arbitrary separating vertex sets.
\begin{Theorem}
\label{thm:splittingUvector}
  Let $(G^1,G^2,X)$ be a splitting of $G$. We have
  \begin{align}
    D(G,x)&=\uvec_X(G^1)^T\Qmat_X\uvec_X(G^2).
  \end{align}
\end{Theorem}
\begin{proof}
  We have, by Lemma~\ref{lem:glueingLemmaDomination}, that
   \begin{align*}
     x^{|W|}&=x^{|W^1\setminus X|}x^{|(W^1\cap X) \cup (W^2\cap X)|}x^{|W^2\setminus X|} 
   \end{align*}
   and 
   \begin{align*}
     [N_G[W]=V]=[N^1\setminus X=V^1\setminus X][B^1\cup B^2=X][N^2\setminus X=V^2\setminus X]
   \end{align*}
   for all $W\subseteq V$ with $N^i:=N_{G^i}[W^i]$, $B^i:=N^i\cap X$ and $i=1,2$.  A
   summation over all $W\subseteq V$ and sorting according to
   Definition~\ref{def:statePolynomialsDominationPolynomial} by
   $(A,B^i)\in R(X)$ yields
  \begin{align*}
    D(G,x)=
    \sum_{\mathclap{\substack{(A,B^1)\in R(X)\\(A,B^2)\in R(X)}}} 
    D(G^1,(A,B^1),x)x^{|A|}[B^1\cup B^2=X] 
    D(G^2,(A,B^2),x)\,,
  \end{align*}
  which is just the stated matrix equation.
\end{proof}

\begin{Definition}
  Let $G=(V,E)$ be a graph with splitting $(G^1,G^2,X)$. For every $(A,B)\in R(X)$ we
  denote by $G_{(A,B)}^i$ the graph
  \begin{align*}
    G_{(A,B)}^i&=G^i+\sum_{\substack{v\in A}}\{v,\cdot\}-\sum_{\substack{v\in B\setminus A }} v\,.
  \end{align*}
  Furthermore we introduce the vectors
  \begin{align*}
    \dvec_X(G^i)&=(D(G_{(A,B)}^i,x))_{(A,B)\in R(X)}
  \end{align*}
  with $\dvec_X(G^i)\in\Mat(R(X))$ for $i=1,2$. These vectors are again extensions
  of the vectors $\dvec_v(G^i)\in\Mat(R(\{v\}))$ already defined in Section~\ref{sec:articulations}.
\end{Definition}

\begin{Proposition}\label{prop:convertuInDVector}
  Let $G=(V,E)$ be a graph and $\emptyset\subset X\subseteq V$. We then have that
  \begin{align*}
    \dvec_X(G)&=\Pmat_X\Qmat_X \uvec_X(G).
  \end{align*}
\end{Proposition}
\begin{proof}
  The proof can be accomplished by utilizing Theorem~\ref{thm:splittingUvector} and is 
  a straightforward extension of Theorem~\ref{thm:utodvector} to arbitrary separating vertex sets.

  Let $(A,B)\in R(X)$ and consider the splitting $(G^1_{(A,B)},G^2_{(A,B)},X)$ of the
  graph
  \begin{align*}
    G_{(A,B)}&=G+\sum_{v\in A}\{v,\cdot\}-\sum_{v\in B\setminus A} v\,,
  \end{align*}
  such that 
  \begin{align*}
    G^1_{(A,B)}&=(X\setminus(B\setminus A)\cup \{v':v\in A\}, \{\{v,v'\}:v\in A\}) \\
    G^2_{(A,B)}&=(V\setminus(B\setminus A), \{\{u,v\}\in E: \{u,v\}\subseteq V\setminus(B\setminus A)\})\,.
  \end{align*}
  Now we have
  \begin{align*}
    D(G_{(A,B)},x)&=\uvec_X(G^1_{(A,B)})^T\Qmat_X\uvec_X(G^2_{(A,B)}) 
  \end{align*}
  by Theorem~\ref{thm:splittingUvector}.  Hence we have to show that
  $\uvec_X(G^1_{(A,B)})$ equals the row of the matrix $\Pmat_X$
  indexed by $(A,B)$. Let $X=\{x_1,\ldots, x_k\}$ and observe that the
  graph $G^1_{(A,B)}$ can be written as the disjoint union of $k$
  graphs, say $H_1,\ldots, H_k$, some of them possibly being null
  graphs (graphs with no vertices), whenever $x_i\in B\setminus A$.  If we pick an element in
  $\uvec_X(G^1_{(A,B)})$ being indexed by $(C,D)\in R(X)$ we can
  therefore conclude
  \begin{align}\label{eq:ABCD}
    D(G_{(A,B)}^1,(C,D),x)&=\prod_{i=1}^k D(H_i, (C\cap \{x_i\}, D\cap \{x_i\}), x)\,,
  \end{align}
  because the domination property is multiplicative in components. The
  right-hand side of Equation (\ref{eq:ABCD}) equals the element of the
  matrix $\Pmat_X$ in the row indexed by $(A,B)$ and the column
  indexed by $(C,D)$. 
\end{proof}

We are now able to state the main theorem of this section which
will enables calculation of the domination polynomial
of a graph given arbitrary separating vertex sets.
\begin{Theorem}[Splitting Formula Theorem] 
\label{thm:symmDomSplittingFormula}
  Let $G=(V,E)$ be a graph with splitting $(G^1,G^2,X)$. We have 
  \begin{align*}
    D(G,x)&=\dvec_X(G^1)^T \Dmat^{-1}_X\dvec_X(G^2)\,.
  \end{align*}
\end{Theorem}
\begin{proof}
  It is
  \begin{align*} 
    D(G,x)&=\uvec_X(G^1)^T \Qmat_X \uvec_X(G^2)  & \tag{Theorem~\ref{thm:splittingUvector}}   \\
    \dvec_X(G^i)&=\Pmat_X\Qmat_X\uvec_X(G^i) \qquad\qquad i=1,2    & \tag{Proposition~\ref{prop:convertuInDVector}}
  \end{align*}
  and by Proposition~\ref{prop:matrixInvertible}  we know that the
  matrices $\Pmat_X$ and $\Qmat_X$ are invertible, so that
  \begin{align*}
    \uvec_X(G^i)&=\Qmat^{-1}_X\Pmat^{-1}_X\dvec_X(G^i)
  \end{align*}
  for $i=1,2$.
  Hence, after substitution, we have 
  \begin{align*}
    D(G,x)&=\dvec_X(G^1)_X^T\Pmat^{-1,T}_X\Qmat^{-1}_X\Pmat^{-1}_X\dvec_X(G^2) \\
    &=\dvec_X(G^1)^T(\Pmat_X\Qmat_X\Pmat^{T}_X)^{-1}\dvec_X(G^2) \\
    &=\dvec_X(G^1)^T\Dmat^{-1}_X\dvec_X(G^2). \qedhere
  \end{align*}
\end{proof}


\subsection{An particular case of the Splitting Formula Theorem} \label{subse:application}
In this subsection we illustrate the use 
of Theorem \ref{thm:symmDomSplittingFormula} by applying it to 
the case of a simple separating set of size $2$. 
Let $G$ be a graph with splitting $(G^1,G^2,X)$ at the
separating vertex set $X=\{u,v\}$. In this case we have the two
matrices $\Dmat_u\in\Mat(R(\{u\}),R(\{u\}))$ and
$\Dmat_v\in\Mat(R(\{v\}),R(\{v\}))$ with
\begin{align*}
  \Dmat_u&=
  \bordermatrix{
    & (\varnothing,\{u\}) & (\varnothing, \varnothing) & (\{u\},\{u\}) \cr
    (\varnothing,\{u\} & 1 & 1 & x\cr
    (\varnothing, \varnothing) & 1 & x & 2x+x^2 \cr
    (\{u\},\{u\}) & x & 2x+x^2 & x+3x^2+x^3 
} \\
  \Dmat_v&=
  \bordermatrix{
    & (\varnothing,\{v\}) & (\varnothing, \varnothing) & (\{v\},\{v\}) \cr
    (\varnothing,\{v\}) & 1 & 1 & x\cr
    (\varnothing, \varnothing) & 1 & x & 2x+x^2 \cr
    (\{v\},\{v\}) & x & 2x+x^2 & x+3x^2+x^3 
}
\end{align*}  
and the inverse matrices
\begin{align*}
  \Dmat_u^{-1}&=
  \bordermatrix{
    & (\varnothing,\{u\}) & (\varnothing, \varnothing) & (\{u\},\{u\}) \cr
    (\varnothing,\{u\}) & x^{3} + 3 x^{2} & x^{2} + x & -2 x \cr
    (\varnothing, \varnothing) & x^{2} + x & -x^{3} - 2 x^{2} - x & x^{2} + x \cr
    (\{u\},\{u\}) & -2 x & x^{2} + x & 1-x
} \cdot
  \frac{1}{x(1+x)^2} \\
  \Dmat_v^{-1}&=
  \bordermatrix{
    & (\varnothing,\{v\}) & (\varnothing, \varnothing) & (\{v\},\{v\}) \cr
    (\varnothing,\{v\}) & x^{3} + 3 x^{2} & x^{2} + x & -2 x \cr
    (\varnothing, \varnothing) & x^{2} + x & -x^{3} - 2 x^{2} - x & x^{2} + x \cr
    (\{v\},\{v\}) & -2 x & x^{2} + x & 1-x
} \cdot
  \frac{1}{x(1+x)^2}\,.
\end{align*}
\newcommand{\pa}{\ensuremath{\scriptstyle{D(P_0)}}}
\newcommand{\pb}{\ensuremath{\scriptstyle{D(P_1)}}}
\newcommand{\pc}{\ensuremath{\scriptstyle{D(P_2)}}}
\newcommand{\pd}{\ensuremath{\scriptstyle{D(P_3)}}}
\newcommand{\pe}{\ensuremath{\scriptstyle{D(P_4)}}} 
By the definition of the Kronecker product,
the matrix $\Dmat_X=\Dmat_u\otimes \Dmat_v$ evaluates 
to the matrix depicted in 
appendix A. 
Observe that the rows and columns of the
matrix $\Dmat_X$ are labelled by $R(\{u\})\times R(\{v\})$ which will be 
somewhat inconvenient. Hence we assume
that for the labelling of the rows and columns the function $f\colon
R(\{u\})\times R(\{v\})\rightarrow R(\{u,v\})$ is applied. This gives
the second matrix in Appendix A. 
The inverse matrix $\Dmat_X^{-1}$ can then be
computed in the same vein by $\Dmat_X^{-1}=\Dmat_u^{-1}\otimes
\Dmat_v^{-1}$ and is shown in Appendix B. 

The vectors $\dvec_X(G^1)$ and $\dvec_X(G^2)$ are given by 
\begin{gather}
\notag
  \dvec_X(G^i)
=
  \bordermatrix{ & \cr
    \scriptstyle{(\varnothing, \{u,v\})}    & D(G^i-u-v,x) \cr
    \scriptstyle{(\varnothing,\{v\})}       & D(G^i-v,x)\cr
    \scriptstyle{(\{u\},\{u,v\})}           & D(G^i+\{u,\cdot\}-v,x)\cr
    \scriptstyle{(\varnothing, \{u\})}      & D(G^i-u,x)\cr
    \scriptstyle{(\varnothing,\varnothing)} & D(G^i,x) \cr
    \scriptstyle{(\{u\},\{u\})}             & D(G^i+\{u,\cdot\},x) \cr
    \scriptstyle{(\{v\},\{u,v\})}           & D(G^i+\{v,\cdot\}-u,x)\cr
    \scriptstyle{(\{v\},\{v\})}             & D(G^i+\{v,\cdot\},x)\cr 
    \scriptstyle{(\{u,v\},\{u,v\})}         & D(G^i+\{u,\cdot\}+\{v,\cdot\},x)\cr
  } 
\end{gather}
and the splitting theorem yields
\begin{align*}
  D(G,x)&=\dvec_X(G^1)^{T} \Dmat_X^{-1}\dvec_X(G^2).
\end{align*}

\begin{thm}
Let $G=(V,E)$ be a graph. For every $e=\{u,v\}\in E$, 
\begin{eqnarray*}
  D(G,x)&=&\frac{1}{(1+x)^2}\Big[ x(1+x)\left(D(G-v,x)+D(G-u,x)\right) \\
& & +(1-x)D(G+\{u,\cdot\}-v,x)-(1+x)D(G-e+\{u,\cdot\},x)\\
& & +(1-x)D(G-u+\{v,\cdot\},x)-(1+x)D(G-e+\{v,\cdot\},x) \\   
& & +(1+x)^2D(G-e,x)+2D(G-e+\{u,\cdot\}+\{v,\cdot\},x) \\ 
& & -2xD(G-u-v,x) \Big]\,.
\end{eqnarray*}
\end{thm}

\begin{proof}
Assume now that 
$e=\{u,v\}\in E$ is an edge of $G=(V,E)$ and that we have the
splitting $G^1=G-e$ and ${G^2=(\{u,v\},\{\{u,v\}\})}$ at the separating
vertex set $X=\{u,v\}$. The vector $\dvec_X(G^2)$ can be computed to be
\begin{align*}
  \dvec_X(G^2)&=
  \bordermatrix{ & \cr
    \scriptstyle{(\varnothing, \{u,v\})}    & D(P_0,x) \cr
    \scriptstyle{(\varnothing,\{v\})}       & D(P_1,x)\cr
    \scriptstyle{(\{u\},\{u,v\})}           & D(P_2,x)\cr
    \scriptstyle{(\varnothing, \{u\})}      & D(P_1,x)\cr
    \scriptstyle{(\varnothing,\varnothing)} & D(P_2,x) \cr
    \scriptstyle{(\{u\},\{u\})}             & D(P_3,x) \cr
    \scriptstyle{(\{v\},\{u,v\})}           & D(P_2,x)\cr
    \scriptstyle{(\{v\},\{v\})}             & D(P_3,x)\cr 
    \scriptstyle{(\{u,v\},\{u,v\})}         & D(P_4,x)\cr
  }=
  \bordermatrix{ & \cr
    \scriptstyle{(\varnothing, \{u,v\})}    & 1 \cr
    \scriptstyle{(\varnothing,\{v\})}       & x\cr
    \scriptstyle{(\{u\},\{u,v\})}           & x^2+2x\cr
    \scriptstyle{(\varnothing, \{u\})}      & x \cr
    \scriptstyle{(\varnothing,\varnothing)} & x^2+2x \cr
    \scriptstyle{(\{u\},\{u\})}             & x^3+3x^2+x \cr
    \scriptstyle{(\{v\},\{u,v\})}           & x^2+2x\cr
    \scriptstyle{(\{v\},\{v\})}             & x^3+3x^2+x\cr 
    \scriptstyle{(\{u,v\},\{u,v\})}         & x^4+4x^3+4x^2\cr
  }
\end{align*}
and a computer algebra system can be used to obtain:
\begin{align*}
  \Dmat_X^{-1}\dvec_X(G^2)&=
  \bordermatrix{ & \cr
    \scriptstyle{(\varnothing, \{u,v\})}    & (1+x)^2 \cr
    \scriptstyle{(\varnothing,\{v\})}       & x(1+x)\cr
    \scriptstyle{(\{u\},\{u,v\})}           & -(1+x)\cr
    \scriptstyle{(\varnothing, \{u\})}      & x(1+x) \cr
    \scriptstyle{(\varnothing,\varnothing)} & -2x \cr
    \scriptstyle{(\{u\},\{u\})}             & 1-x \cr
    \scriptstyle{(\{v\},\{u,v\})}           & -(1+x)\cr
    \scriptstyle{(\{v\},\{v\})}             & 1-x\cr 
    \scriptstyle{(\{u,v\},\{u,v\})}         & 2\cr
  } \cdot \frac{1}{(1+x)^2}\,.
\end{align*}
Hence we compute by the splitting theorem
\begin{eqnarray*} 
  D(G,x)&=& \frac{1}{(1+x)^2} \Big[ x(1+x)\left(D(G^1-v,x)+D(G^1-u,x)\right)+\\
& & (1-x)D(G^1+\{u,\cdot\}-v,x)-(1+x)D(G^1+\{u,\cdot\},x)+\\
  &&(1-x)D(G^1-u+\{v,\cdot\},x)-(1+x)D(G^1+\{v,\cdot\},x)+\\
&&(1+x)^2D(G^1,x)+2D(G^1+\{u,\cdot\}+\{v,\cdot\},x)\\
& & -2xD(G^1-u-v,x) \Big].
\end{eqnarray*}

Substituting $G^1= G - e$ gives the required answer.
\end{proof}

It is interesting that this formula, like 
Theorem \ref{th:arbitrary_rec}, is a linear
combination of exactly nine terms.  Unfortunately, the formula is not
as useful as Theorem \ref{th:arbitrary_rec}, since we have to compute domination
polynomials of graphs that are emerging from $G$ by adding additional
edges.


\section{A recurrence relation using derivatives}\label{se:der}
\begin{lem}
Let $G=(V,E)$ be a graph and $u\in V$.  Then
\label{lem:D_w} $D_{u\in W}(G,x)=D(G,x)-D(G-u,x)+\p{u}(G,x)$\,.
\end{lem}
\begin{proof}
From Theorem \ref{t:red} we have
\begin{gather} \label{eq:notin}
 D_{u \notin W}(G, x) = D(G-u, x) - \p{u}(G, x)\,.
\end{gather}
Adding $D_{u \in W}(G,x)$ to both sides of Equation (\ref{eq:notin})
and rearranging the terms, we get the lemma.
\end{proof}

We denote by 
$A(G,x)$ the following sum: 
\[
A(G,x)=\sum_{v\in V}D(G-v,x)-D(G/v,x)-D(G-N[v],x)\,.
\]

$D^{(i)}(G,x)$ denotes the $i$th derivative of the domination polynomial
$D(G,x)$ with respect to its indeterminate $x$. Similarly, $A^{(i)}(G,x)$
denotes the $i$th derivative of $A(G,x)$. 
\begin{thm}
\label{th:der}
Let $G=(V,E)$ be a graph. For every $i\geq0$, 
\[
D^{(i)}(G,x)=\frac{1+x}{|V|-i}D^{(i+1)}(G,x)+\frac{1}{|V|-i}A^{(i)}(G,x)\,.
\]

\begin{proof}
The proof is by induction. First we prove the case for $i=0$. 
\begin{eqnarray*}
D(G,x) & = & \sum_{W\subseteq V:N[W]=V}x^{|W|}\\
 & = & \sum_{W\subseteq V:N[W]=V}\,\sum_{w\in W}\frac{1}{|W|}x^{|W|}\\
 & = & \sum_{w\in V}\,\sum_{\{w\}\subseteq W\subseteq V:N[W]=V}\frac{1}{|W|}x^{|W|}\,.
\end{eqnarray*}
Taking the derivative of $D(G,x)$ with respect to $x$ and using Lemma \ref{lem:D_w}
we get
\begin{eqnarray*}
D^{(1)}(G,x) 
& = & \sum_{w\in V}\sum_{\{w\}\subseteq W\subseteq V:N[W]=V}x^{|W|-1}\\
& = & \sum_{w\in V} \frac{1}{x} D_{w\in W}(G,x)\\
 & = & \frac{1}{x}\sum_{w\in V}D(G,x)-D(G-w,x)+\p{w}(G,x)\,.
\end{eqnarray*}
Now, using Theorem \ref{t:red}, 
\begin{eqnarray*}
D^{(1)}(G,x) & = & \frac{1}{1+x}\sum_{w\in V} \big( D(G,x)-D(G-w,x) + D(G/w,x)
\\ & &  +D(G-N[w],x)\big)
  =  \frac{|V|}{1+x}D(G,x)-\frac{1}{1+x}A(G,x)\,,
\end{eqnarray*}
and the case $i=0$ follows.

Assuming that we have 
$D^{(i)}(G,x)$ as stated in the theorem we can
multiply by $|V|-i$ and take the derivative of both sides to get
\[
(|V|-i) D^{(i+1)}(G,x)=(1+x)D^{(i+2)}(G,x)+
D^{(i+1)}(G,x)+ A^{(i+1)}(G,x) \,.
\]
We can then re-arrange this equation to get
\[
(|V|-(i+1)) D^{(i+1)}(G,x)=(1+x)D^{(i+2)}(G,x)+ A^{(i+1)}(G,x) 
\]
which establishes the induction on division by $|V|- (i+1)$.
\end{proof}
\end{thm}

We can obtain a representation of $D(G,x)$
in terms of $A(G,x)$.
\begin{thm}
Let $G=(V,E)$ be a graph. Then
\[
D(G,x)=(1+x)^{|V|}+\sum_{i=0}^{|V|-1}(1+x)^{i}\frac{\left(|V|-(i+1)\right)!}{|V|!}A^{(i)}(G,x)\,.
\]
\end{thm}

\begin{proof}
Rearranging Theorem \ref{th:der} to make $A^{(i)}(G,x)$ the subject we have that 
\[
A^{(i)}(G,x) = (|V|-i) D^{(i)}(G,x)-(1+x)D^{(i+1)}(G,x)\,.
\]
All but two terms of this sum cancel by combining adjacent terms:
\begin{eqnarray*}
\sum_{i=0}^{|V|-1}(1+x)^{i}\frac{\left(|V|-(i+1)\right)!}{|V|!}&A^{(i)}(G,x)&
= \frac{|V| D(G,x) - (1+x) D^{(1)} (G,x)}{|V|} \\
&+& (1+x) \frac{ (|V|-1) D^{(1)}(G,x) - (1+x) D^{(2)} (G,x) }{|V|(|V|-1)}  \\
&+& \ldots \\
&+& (1+x)^{|V|-1} \frac{ D^{(|V|-1)}(G,x) - (1+x) D^{|V|} (G,x) }{|V|!}  \\
&=& D(G,x) - (1+x)^{|V|}\,.
\end{eqnarray*}
We use that 
$D(G,x)$ is a monic polynomial of degree $|V|$ giving $D^{|V|} (G,x) =|V|!$
\end{proof}

\section{Conclusion}

The domination polynomial resembles such graph polynomials as the independence polynomial $I(G,x)$ and the vertex cover polynomial $VC(G,x)$
insofar as they are all defined as generating functions of certain subsets of vertices. 
However, we showed that the domination polynomial has quite a different behavior with respect to recurrence relations.
Theorems \ref{th:non-exist_v} and \ref{th:non-exist_e} show that $D(G,x)$ cannot satisfy a linear recurrence relation 
analogous to $I(G,x)$, $VC(G,x)$ and other prominent graph polynomials.

We gave many recurrence relations and splitting formulas for the domination polynomial.
Theorem \ref{t:red} gives a reduction formula for $D(G,x)$ based on the related $\p{u}(G,x)$ polynomial.
This theorem gives rise to recurrence steps in various special cases, as well as to 
a linear recurrence relation for arbitrary graphs which uses compositions of standard edge and vertex elimination operations, Theorem \ref{th:arbitrary_rec}. 
We gave splitting formulas for $D(G,x)$ in the case that it is $1$-connected. We then generalized
this in Theorem~\ref{thm:symmDomSplittingFormula} to a splitting formula which allows 
to compute the domination polynomial of a graph $G$, by
separating the graph into two parts, so that we have to compute
domination polynomials of the modifications of the resulting subgraphs
$G^1$ and $G^2$.
Finally, we gave a rather simple recurrence relation for 
$D(G,x)$ using derivatives of domination polynomials of smaller graphs
in Theorem \ref{th:der}.

The domination polynomial is therefore established to have surprisingly diverse and unique decomposition formulas
and warrants further research. 
The paper leaves some open problems, among them:

\begin{openproblem}
Are there simple graph operations which can be used to give
a simpler recurrence relation for the domination polynomial of arbitrary graphs?
\end{openproblem}

In Theorem~\ref{thm:symmDomSplittingFormula} it is 
necessary to attach additional edges to the
subgraphs $G^1$ and $G^2$ in order to state this formula. 
\begin{openproblem}
Is it possible to give a splitting formula in the vein of Theorem \ref{thm:symmDomSplittingFormula} 
which avoids adding edges?
\end{openproblem}

Another interesting observation is that some edges might be shifted to
different positions in a graph, without changing the resulting
domination polynomial. 
\begin{openproblem}
Is there a criterion which characterizes the edges which 
can be shifted without changing $D(G,x)$?
\end{openproblem}



 \nocite{Alikhani2009}
 \nocite{Arocha2000}
 \nocite{Akbari2010}
 \nocite{Frucht1970}
 \nocite{Dohmen2012}
 \nocite{ar:Alikhani2011ARX}
 \nocite{ar:Alikhani2011} 
 \nocite{ar:AlikhaniPeng2011}
 \nocite{ar:AkbariAlikhaniOboudiPeng2010}
 \nocite{ar:AlikhaniPengIntro}

\newpage
\begin{sidewaysfigure}
\centerline{ Appendix A: $\Dmat_X$ with $X=\{u,v\}$}\label{app:dmatx}
{\scriptsize
\begin{align*}
  \bordermatrix{ &&&&&&&&&
    \cr \vspace{5pt}
    \scriptstyle{((\varnothing,\{u\}),(\varnothing,\{v\}))} 
    & \nm{\pa\pa} & \nm{\pb\pa} & \nm{\pc\pa} & \nm{\pa\pb} & \nm{\pb\pb} & \nm{\pc\pb} & \nm{\pa\pc} & \nm{\pb\pc} & \nm{\pc\pc} \cr \vspace{5pt}
    \scriptstyle{((\varnothing, \varnothing),(\varnothing,\{v\}))} 
    & \nm{\pb\pa} & \nm{\pc\pa} & \nm{\pb\pa} & \nm{\pb\pb} & \nm{\pc\pb} & \nm{\pd\pb} & \nm{\pb\pc} & \nm{\pc\pc} & \nm{\pd\pc} \cr \vspace{5pt}
    \scriptstyle{((\{u\},\{u\}),(\varnothing,\{v\}))} 
    & \nm{\pc\pa} & \nm{\pb\pa} & \nm{\pe\pa} & \nm{\pc\pb} & \nm{\pd\pb} & \nm{\pe\pb} & \nm{\pc\pc} & \nm{\pd\pc} & \nm{\pe\pc} \cr \vspace{5pt}
    \scriptstyle{((\varnothing,\{u\}),(\varnothing,\varnothing))} 
    & \nm{\pa\pb} & \nm{\pb\pb} & \nm{\pc\pb} & \nm{\pa\pc} & \nm{\pb\pc} & \nm{\pc\pc} & \nm{\pa\pd} & \nm{\pb\pd} & \nm{\pc\pd} \cr \vspace{5pt}
    \scriptstyle{((\varnothing, \varnothing),(\varnothing,\varnothing))} 
    & \nm{\pb\pb} & \nm{\pc\pb} & \nm{\pd\pb} & \nm{\pb\pc} & \nm{\pc\pc} & \nm{\pd\pc} & \nm{\pb\pd} & \nm{\pc\pd} & \nm{\pd\pd} \cr \vspace{5pt}
    \scriptstyle{((\{u\},\{u\}),(\varnothing,\varnothing))} 
    & \nm{\pc\pb} & \nm{\pd\pb} & \nm{\pe\pb} & \nm{\pc\pc} & \nm{\pd\pc} & \nm{\pe\pc} & \nm{\pc\pd} & \nm{\pd\pd} & \nm{\pe\pd} \cr \vspace{5pt}
    \scriptstyle{((\varnothing,\{u\}),(\{v\},\{v\}))} 
    & \nm{\pa\pc} & \nm{\pb\pc} & \nm{\pc\pc} & \nm{\pa\pd} & \nm{\pb\pd} & \nm{\pc\pd} & \nm{\pa\pe} & \nm{\pb\pe} & \nm{\pc\pe} \cr \vspace{5pt}
    \scriptstyle{((\varnothing,\varnothing),(\{v\},\{v\}))} 
    & \nm{\pb\pc} & \nm{\pc\pc} & \nm{\pd\pc} & \nm{\pb\pd} & \nm{\pc\pd} & \nm{\pd\pd} & \nm{\pb\pe} & \nm{\pc\pe} & \nm{\pd\pe} \cr \vspace{5pt}
    \scriptstyle{((\{u\},\{u\}),(\{v\},\{v\}))}  
    & \nm{\pc\pc} & \nm{\pd\pc} & \nm{\pe\pc} & \nm{\pc\pd} & \nm{\pd\pd} & \nm{\pe\pd} & \nm{\pc\pe} & \nm{\pd\pe} & \nm{\pe\pe} \cr 
  }
\end{align*}
}
%
{\small
\begin{align*}
  \bordermatrix{
    & 
    \scriptstyle{(\varnothing, \{u,v\})} & 
    \scriptstyle{(\varnothing, \{v\})} & 
    \scriptstyle{(\{u\},\{u,v\})} &
    \scriptstyle{(\varnothing, \{u\})} & 
    \scriptstyle{(\varnothing,\varnothing)} & 
    \scriptstyle{(\{u\},\{u\})} & 
    \scriptstyle{(\{v\},\{u, v\})} & 
    \scriptstyle{(\{v\},\{v\})} & 
    \scriptstyle{(\{u,v\},\{u,v\})}  \cr
    \scriptstyle{(\varnothing,\{u,v\})} & \pa\pa & \pb\pa & \pc\pa & \pa\pb & \pb\pb & \pc\pb & \pa\pc & \pb\pc & \pc\pc \cr
    \scriptstyle{(\varnothing, \{v\})}      & \pb\pa & \pc\pa & \pd\pa & \pb\pb & \pc\pb & \pd\pb & \pb\pc & \pc\pc & \pd\pc \cr
    \scriptstyle{(\{u\},\{u,v\})}             & \pc\pa & \pd\pa & \pe\pa & \pc\pb & \pd\pb & \pe\pb & \pc\pc & \pd\pc & \pe\pc \cr
    \scriptstyle{(\varnothing,\{u\})}       & \pa\pb & \pb\pb & \pc\pb & \pa\pc & \pb\pc & \pc\pc & \pa\pd & \pb\pd & \pc\pd \cr
    \scriptstyle{(\varnothing, \varnothing)}    & \pb\pb & \pc\pb & \pd\pb & \pb\pc & \pc\pc & \pd\pc & \pb\pd & \pc\pd & \pd\pd \cr
    \scriptstyle{(\{u\},\{u\})}           & \pc\pb & \pd\pb & \pe\pb & \pc\pc & \pd\pc & \pe\pc & \pc\pd & \pd\pd & \pe\pd \cr
    \scriptstyle{(\{v\},\{u,v\})}             & \pa\pc & \pb\pc & \pc\pc & \pa\pd & \pb\pd & \pc\pd & \pa\pe & \pb\pe & \pc\pe  \cr 
    \scriptstyle{(\{v\},\{v\})}           & \pb\pc & \pc\pc & \pd\pc & \pb\pd & \pc\pd & \pd\pd & \pb\pe & \pc\pe & \pd\pe \cr
    \scriptstyle{(\{u,v\},\{u,v\})}         & \pc\pc & \pd\pc & \pe\pc & \pc\pd & \pd\pd & \pe\pd & \pc\pe & \pd\pe & \pe\pe \cr
  }
\end{align*}
}
\end{sidewaysfigure}

\newpage
\begin{sidewaysfigure}
\centerline{Appendix B: $x^2(1+x)^2 \Dmat_X^{-1}$ with $X=\{u,v\}$}\label{app:dmatx3}
{\tiny
{\small
\begin{align*}
\left(\begin{array}{rrrrrrrrr}
\scriptstyle{\left(x + 3\right)}^{2} x^{4} &\scriptstyle {\left(x + 1\right)} {\left(x +
3\right)} x^{3} &\scriptstyle -2 \, {\left(x + 3\right)} x^{3} &\scriptstyle {\left(x +
1\right)} {\left(x + 3\right)} x^{3} &\scriptstyle {\left(x + 1\right)}^{2}
x^{2} &\scriptstyle -2 \, {\left(x + 1\right)} x^{2} &\scriptstyle -2 \, {\left(x +
3\right)} x^{3} &\scriptstyle -2 \, {\left(x + 1\right)} x^{2} &\scriptstyle 4 \, x^{2}
\\
\scriptstyle{\left(x + 1\right)} {\left(x + 3\right)} x^{3} &\scriptstyle -{\left(x +
1\right)}^{2} {\left(x + 3\right)} x^{3} &\scriptstyle {\left(x + 1\right)}
{\left(x + 3\right)} x^{3} &\scriptstyle {\left(x + 1\right)}^{2} x^{2} &\scriptstyle
-{\left(x + 1\right)}^{3} x^{2} &\scriptstyle {\left(x + 1\right)}^{2} x^{2}
&\scriptstyle -2 \, {\left(x + 1\right)} x^{2} &\scriptstyle 2 \, {\left(x +
1\right)}^{2} x^{2} &\scriptstyle -2 \, {\left(x + 1\right)} x^{2} \\
\scriptstyle-2 \, {\left(x + 3\right)} x^{3} &\scriptstyle {\left(x + 1\right)} {\left(x +
3\right)} x^{3} &\scriptstyle -{\left(x - 1\right)} {\left(x + 3\right)} x^{2}
&\scriptstyle -2 \, {\left(x + 1\right)} x^{2} &\scriptstyle {\left(x + 1\right)}^{2}
x^{2} &\scriptstyle -{\left(x - 1\right)} {\left(x + 1\right)} x &\scriptstyle 4 \,
x^{2} &\scriptstyle -2 \, {\left(x + 1\right)} x^{2} &\scriptstyle 2 \, {\left(x -
1\right)} x \\
\scriptstyle{\left(x + 1\right)} {\left(x + 3\right)} x^{3} &\scriptstyle {\left(x +
1\right)}^{2} x^{2} &\scriptstyle -2 \, {\left(x + 1\right)} x^{2} &\scriptstyle
-{\left(x + 1\right)}^{2} {\left(x + 3\right)} x^{3} &\scriptstyle -{\left(x +
1\right)}^{3} x^{2} &\scriptstyle 2 \, {\left(x + 1\right)}^{2} x^{2} &\scriptstyle
{\left(x + 1\right)} {\left(x + 3\right)} x^{3} &\scriptstyle {\left(x +
1\right)}^{2} x^{2} &\scriptstyle -2 \, {\left(x + 1\right)} x^{2} \\
\scriptstyle{\left(x + 1\right)}^{2} x^{2} &\scriptstyle -{\left(x + 1\right)}^{3} x^{2}
&\scriptstyle {\left(x + 1\right)}^{2} x^{2} &\scriptstyle -{\left(x + 1\right)}^{3}
x^{2} &\scriptstyle {\left(x + 1\right)}^{4} x^{2} &\scriptstyle -{\left(x +
1\right)}^{3} x^{2} &\scriptstyle {\left(x + 1\right)}^{2} x^{2} &\scriptstyle -{\left(x
+ 1\right)}^{3} x^{2} &\scriptstyle {\left(x + 1\right)}^{2} x^{2} \\
\scriptstyle-2 \, {\left(x + 1\right)} x^{2} &\scriptstyle {\left(x + 1\right)}^{2} x^{2}
&\scriptstyle -{\left(x - 1\right)} {\left(x + 1\right)} x &\scriptstyle 2 \, {\left(x +
1\right)}^{2} x^{2} &\scriptstyle -{\left(x + 1\right)}^{3} x^{2} &\scriptstyle {\left(x
- 1\right)} {\left(x + 1\right)}^{2} x &\scriptstyle -2 \, {\left(x + 1\right)}
x^{2} &\scriptstyle {\left(x + 1\right)}^{2} x^{2} &\scriptstyle -{\left(x - 1\right)}
{\left(x + 1\right)} x \\
\scriptstyle{ -2 \, {\left(x + 3\right)} x^{3}} &\scriptstyle -2 \, {\left(x + 1\right)} x^{2}
&\scriptstyle 4 \, x^{2} &\scriptstyle {\left(x + 1\right)} {\left(x + 3\right)} x^{3}
&\scriptstyle {\left(x + 1\right)}^{2} x^{2} &\scriptstyle -2 \, {\left(x + 1\right)}
x^{2} &\scriptstyle -{\left(x - 1\right)} {\left(x + 3\right)} x^{2} &\scriptstyle
-{\left(x - 1\right)} {\left(x + 1\right)} x &\scriptstyle 2 \, {\left(x -
1\right)} x \\
\scriptstyle-2 \, {\left(x + 1\right)} x^{2} &\scriptstyle 2 \, {\left(x + 1\right)}^{2}
x^{2} &\scriptstyle -2 \, {\left(x + 1\right)} x^{2} &\scriptstyle {\left(x +
1\right)}^{2} x^{2} &\scriptstyle -{\left(x + 1\right)}^{3} x^{2} &\scriptstyle {\left(x
+ 1\right)}^{2} x^{2} &\scriptstyle -{\left(x - 1\right)} {\left(x + 1\right)} x
&\scriptstyle {\left(x - 1\right)} {\left(x + 1\right)}^{2} x &\scriptstyle -{\left(x -
1\right)} {\left(x + 1\right)} x \\
\scriptstyle4 \, x^{2} &\scriptstyle -2 \, {\left(x + 1\right)} x^{2} &\scriptstyle 2 \, {\left(x -
1\right)} x &\scriptstyle -2 \, {\left(x + 1\right)} x^{2} &\scriptstyle {\left(x +
1\right)}^{2} x^{2} &\scriptstyle -{\left(x - 1\right)} {\left(x + 1\right)} x
&\scriptstyle 2 \, {\left(x - 1\right)} x &\scriptstyle -{\left(x - 1\right)} {\left(x +
1\right)} x &\scriptstyle {\left(x - 1\right)}^{2}
\end{array}\right)
\end{align*}}
}
\end{sidewaysfigure}

\end{document}